\documentclass[12pt]{article}

\usepackage{amsmath,amssymb}
\usepackage{amsthm}
\usepackage[pdftex,bookmarks=true]{hyperref}
\usepackage{thmtools}
\usepackage{thm-restate}
\usepackage{cite}
\usepackage{enumitem}

\usepackage{comment}

\usepackage{tikz}
\usetikzlibrary{arrows}
\usetikzlibrary{decorations.pathmorphing}
\usetikzlibrary{decorations.markings}
\tikzset{
every node/.style={draw, circle, inner sep=2pt},
every label/.style={rectangle, draw=none}
}

\usepackage{soul} 
\usepackage{cancel}

\newtheorem{theorem}{Theorem}[section]
\newtheorem{lemma}[theorem]{Lemma}
\newtheorem{proposition}[theorem]{Proposition}
\newtheorem{corollary}[theorem]{Corollary}

\theoremstyle{definition}
\newtheorem{definition}[theorem]{Definition}

\newtheorem{remark}[theorem]{Remark}
\newtheorem{example}[theorem]{Example}

\newtheorem{question}[theorem]{Question}

\newcommand{\bone}{{\bf 1}}
\newcommand{\bzero}{{\bf 0}}
\newcommand{\trans}{^\top}
\newcommand{\dunion}{\mathbin{\dot\cup}}
\newcommand{\bx}{{\bf x}}
\newcommand{\ba}{{\bf a}}
\newcommand{\bb}{{\bf b}}

\newcommand{\bd}{{\bf d}}

\newcommand{\be}{{\bf e}}

\newcommand{\bu}{{\bf u}}
\newcommand{\by}{\mathbf{y}}

\newcommand{\bw}{\mathbf{w}}

\newcommand{\Col}{\operatorname{Col}}

\newcommand{\vspan}{\operatorname{span}}

\newcommand{\spec}{\operatorname{spec}}
\newcommand{\range}{\operatorname{range}}
\newcommand{\nul}{\operatorname{null}}

\newcommand{\tr}{\operatorname{tr}}

\newcommand{\mptn}{\mathcal{S}} 
\newcommand{\mptncl}{\mathcal{S}^{\rm cl}}

\newcommand{\GL}[2][n]{\operatorname{GL}^{(#2)}_{#1}(\mathbb{R})}

\newcommand{\Paw}{\mathrm{Paw}}

\newcommand{\inp}[2]{\left\langle#1,#2\right\rangle}

\newcommand{\mat}[1][n]{\operatorname{Mat}_{#1}(\mathbb{R})}
\newcommand{\msym}[1][n]{\operatorname{Sym}_{#1}(\mathbb{R})}

\title{The inverse nullity pair problem and\\ the strong nullity interlacing property}
\author{
Aida Abiad
\thanks{Department of Mathematics and Computer Science, Eindhoven University of Technology, Eindhoven, The Netherlands. Department of Mathematics: Analysis, Logic and Discrete Mathematics, Ghent University, Ghent, Belgium (a.abiad.monge@tue.nl)}
\and 
Bryan A. Curtis
\thanks{Department of Mathematics, Iowa State University,
Ames, IA 50011, USA (bcurtis1@iastate.edu)}
\and 
Mary Flagg
\thanks{Department of Mathematics, Statistics and Computer Science, University of St.~Thomas, Houston, TX, U.S.A. (flaggm@stthom.edu)}
\and
H. Tracy Hall
\thanks{Hall Labs LLC, Provo Utah, U.S.A. (h.tracy@gmail.com)}
\and
Jephian C.-H.~Lin
\thanks{Department of Applied Mathematics, National Sun Yat-sen University, Kaohsiung 80424, Taiwan (jephianlin@gmail.com)}
\and 
Bryan Shader
\thanks{University of Wyoming, Colorado, U.S.A. (bshader@uwyo.edu)}
}

\date{}

\begin{document}
\maketitle

\begin{abstract}
The inverse eigenvalue problem studies the possible spectra among matrices whose off-diagonal entries have their zero-nonzero patterns described by the adjacency of a graph $G$.  In this paper, we refer to the $i$-nullity pair of a matrix $A$ as $(\nul(A), \nul(A(i))$, where $A(i)$ is the matrix obtained from $A$ by removing the $i$-th row and column. The inverse $i$-nullity pair problem is considered for complete graphs, cycles, and trees.  The strong nullity interlacing property is introduced, and the corresponding supergraph lemma and decontraction lemma are developed as new tools for constructing matrices with a given nullity pair.
\end{abstract}  

 
\medskip


\section{Introduction}
\label{sec:intro}

Inverse eigenvalue problems are interested in the existence of matrices that have prescribed spectral data. Often these matrices have additional restraints. For example, the entries of a matrix can be restricted by the adjacencies in an associated graph. In this paper we study an inverse eigenvalue problem in which we are given a pair of integers $(k,\ell)$ and are asked to find a real symmetric matrix with nullity $k$ for which a certain principal submatrix has nullity $\ell$. Before formalizing this problem, we provide some background and introduce the required notation.

Let $G$ be a simple graph on $n$ vertices. We always label the vertex set $V(G)$ of $G$ with the set $[n] = \{1,\ldots,n\}$ so that there is a natural correspondence between indices of matrices in $\mptn(G)$ and vertices of $G$. The set $\mptn(G)$ contains all real symmetric matrices $A = \begin{bmatrix} a_{i,j}\end{bmatrix}$ such that $a_{i,j} \neq 0$ for $i \neq j$ if and only if $\{i,j\}$ is an edge of $G$. Note that $G$ only governs the off-diagonal entries of matrices in $\mptn(G)$ via its adjacencies and no restrictions are placed on the diagonal entries. An important problem that is closely related to the present work is the \emph{inverse eigenvalue problem of a graph $G$} (IEP-$G$) which aims to find all possible spectra among matrices in $\mptn(G)$.  

One source of inspiration for this paper is the $\lambda,\mu$ problem \cite{lambmu2013}. This variant of the IEP-$G$ studies the interlacing inequalities for a matrix $A$ and its principal submatrix $A(i)$ obtained by deleting row and column $i$. In 1974, Hochstadt~\cite{Hochstadt74} studied the matrices of paths $P_n$ and showed that given some distinct real numbers $\lambda_1 < \mu_1 < \lambda_2 < \cdots < \mu_{n-1} < \lambda_n$, there is at most one matrix $A$ in $\mptn(P_n)$ with nonnegative off-diagonal entries such that $\spec(A) = \{\lambda_1, \ldots, \lambda_n\}$ and $\spec(A(1)) = \{\mu_1, \ldots, \mu_{n-1}\}$. In 1976, Gray and Wilson~\cite{GW76} and Hald~\cite{Hald76} independently gave constructive proofs showing such $A$ always exists. 
Ferguson~\cite{Ferguson80} continued this line of research and studied the spectra of matrices associated with cycles $\mptn(C_n)$.  Monfared and Shader~\cite{MS13} showed that for each connected graph $G$ and a vertex $i\in V(G)$, given  distinct real numbers $\lambda_1 < \mu_1 < \lambda_2 < \cdots < \mu_{n-1} < \lambda_n$, there exists a matrix $A\in\mptn(G)$ such that $\spec(A) = \{\lambda_1, \ldots, \lambda_n\}$ and $\spec(A(i)) = \{\mu_1, \ldots, \mu_{n-1}\}$.  

Another approach to the IEP-$G$ is through the maximum nullity of a graph. The \emph{maximum nullity} of a graph $G$, denoted $M(G)$, is the largest nullity amongst all matrices in $\mptn(G)$. Since $A\in\mptn(G)$ implies $A - \lambda I\in\mptn(G)$, the maximum nullity $M(G)$ is also the largest multiplicity among all eigenvalues of all matrices in $\mptn(G)$. The maximum nullity of a graph has garnered much attention in recent years; see, e.g., \cite{IEPGZF22} and the references therein.  

In this paper, we combine the idea of interlacing from the $\lambda,\mu$ problem with the maximum nullity of a graph. Let $i \in [n]$ and let $A$ be an $n\times n$ real symmetric matrix.  The \emph{$i$-nullity pair} of $A$ is the pair $(\nul(A), \nul(A(i))$. By convention, the nullity of a matrix of order $0$ is considered to be $0$. Let $G$ be a graph and $i\in V(G)$.  We say that $G$  \emph{allows the $i$-nullity pair} $(k,\ell)$ 
provided that there is a matrix $A\in\mptn(G)$ with the $i$-nullity pair $(k,\ell)$. 
Note that the only possible nullity pairs $(k,\ell)$ are those with $|k - \ell| \leq 1$.


\begin{question}
Let $G$ be a graph and $i\in V(G)$.  Given a pair of nonnegative integers $(k,\ell)$, is there a matrix $A\in \mptn(G)$ such that $(\nul(A),\nul(A(i))=(k,\ell)$?
\end{question}
 
Many researchers working on problems related to the IEP-$G$ have had great success using ``strong properties.'' The development of strong properties is due to the pioneering work of Y.~Colin de Verdière while studying the maximum nullity of a special class of matrices. These ideas ultimately lead to the strong Arnold property and many other strong properties have since been studied \cite{CdV, CdVF, BFH05,gSAP, IEPG2}. Motivated by this, we introduce the \emph{strong nullity interlacing property} (SNIP).
\begin{definition}
An $n\times n$ matrix $A$ is said to have the \emph{$i$-strong nullity interlacing property} ($i$-SNIP) if $X = O$ is the only symmetric matrix that satisfies $A\circ X = O$, $I\circ X = O$ and $(AX)(i,:] = O$, where $(AX)(i,:]$ is the submatrix of $AX$ obtained by removing the $i$-th row.
\end{definition}

The theoretical underpinnings of the SNIP are postponed until Section \ref{sec:theory} since they closely resemble the development of other strong properties. Our primary application of the SNIP is Theorem~\ref{thm:minmono} which allows for the characterization of many graphs that allow the nullity pair $(k,\ell)$ in terms of graph minors.

Recall that the \emph{contraction} of a graph $G$ along an edge $\{u,v\}$ is the graph obtained from $G$ by identifying $u$ and $v$ and removing any resulting loops and multi-edges. We write $G - v$ to denote the graph obtained from $G$ by removing the vertex $v$. A \emph{rooted graph} is a pair $(G,i)$, where $G$ is a simple graph and $i\in V(G)$ is called the \emph{root}.  We say that the rooted graph $(G,i)$ is a \emph{rooted minor} of $(H,i)$ if $(G,i)$ can be obtained from $(H,i)$ by a sequence of edge deletions, deletions of isolated vertices $v\not=i$, and edge contractions (the newly formed vertex is the root if the contracted edge is incident to $i$).
Since we only focus on rooted minors, we may refer to a rooted minor simply as a minor.

We say that the rooted graph $(G,i)$ \emph{allows the nullity pair} $(k,\ell)$ (with the SNIP, respectively) provided that there is a matrix $A\in\mptn(G)$ with the $i$-nullity pair $(k,\ell)$ (with the $i$-SNIP, respectively). We say that a rooted graph $(G,i)$ is a \emph{minimal rooted minor} (or a \emph{minimal minor}) for the nullity pair $(k,\ell)$ if $(G,i)$ allows the nullity pair $(k,\ell)$ with the SNIP and each of its proper minors does not allow the nullity pair $(k,\ell)$ with the SNIP. 

Next we state the Minor Monotonicity Theorem for nullity pairs, whose proof will be given in Section~\ref{sec:theory}.  

\begin{restatable}[Minor Monotonicity]{theorem}{minmono}
\label{thm:minmono}
Let $(G,i)$ be a rooted minor of $(H,i)$.  If $(G,i)$ allows the nullity pair $(k,\ell)$ with the SNIP, then $(H,i)$ allows the nullity pair $(k,\ell)$ with the SNIP. 
\end{restatable}

A consequence of this theorem is that a rooted graph $(G,i)$ allows the nullity pair $(k,\ell)$ with the SNIP if and only if $(G,i)$ contains a minimal minor for $(k,\ell)$ as a minor. The importance of this observation cannot be understated as it justifies most of the present work.

This paper is organized as follows. In Section \ref{sec:preliminaries} we establish basic notation, definitions, and background. Then, in Section~\ref{sec:snipsap}, we investigate the behavior of the SNIP and its relation with the strong Arnold property (SAP). In Section~\ref{sec:inpp}, we provide many examples by establishing the realizable nullity pairs for various families of graphs. Illustrating the strength of Theorem \ref{thm:minmono}, Section~\ref{sec:minors} applies the SNIP to characterize the rooted graphs $(G,i)$ that allow the $i$-nullity pairs $(k,\ell)$ with $k<2$ or $\ell < 2$. 
Finally, in Section~\ref{sec:theory}, we provide the proofs of the supergraph and decontraction lemmas as well as the proof of Theorem \ref{thm:minmono}.

\section{Preliminaries}
\label{sec:preliminaries}


All matrices considered in this paper are real matrices. We write $\mat$ and $\msym$ for the set of $n\times n$ matrices and symmetric matrices, respectively.

Recall that $[n]$ indicates the first $n$ positive integers. Let $A$ be an $m\times n$ matrix, $\alpha\subseteq [m]$, and $\beta\subseteq [n]$.  Then $A[\alpha,\beta]$ is the submatrix of $A$ induced by the rows indexed by $\alpha$ and columns indexed by $\beta$.  The submatrix of $A$ obtained by deleting the rows indexed by $\alpha$ and columns indexed by $\beta$ is denoted by $A(\alpha,\beta)$. We may also combine both notational conventions, for example $A(\alpha,\beta]$ is the submatrix of $A$ obtained by removing the rows indexed by $\alpha$ and keeping columns indexed by $\beta$.  When $\alpha = \beta$, we abbreviate our notation to $A[\alpha]$, $A[\alpha)$, $A(\alpha]$, and $A(\alpha)$. When $\alpha = \{i\}$, we abbreviate $A(\{i\})$ by $A(i)$. Following the convention in many programming languages, the symbol $:$ stands for all indices, e.g., $A(i,:]$ is the submatrix of $A$ obtained by deleting row $i$. 

As noted earlier, 
if $(k, \ell)$ is a realizable $i$-nullity pair of $G$, then $|k - \ell| \leq 1$.  Further, there is a matrix $A\in\mptn(G)$ with $\nul(A) = k$ if and only if $0 \leq k \leq M(G)$.  
It is well-known and straightforward to check the following characterization; see, e.g., \cite{BFH04}.  

\begin{remark}
\label{rem:threecases}
Let $A\in\msym$ and $i\in[n]$. Without loss of generality $i = 1$ and so $A$ has the form
\begin{equation}
\label{eq:Apartition}
    A = \begin{bmatrix} a & \bb\trans \\
    \bb & C
    \end{bmatrix},    
\end{equation}
where $C = A(i)$. Then the relations between $\nul(A)$ and $\nul(A(i))$ are as follows:
\begin{enumerate}
\item $\nul(A) + 1 = \nul(A(i))$ if and only if $\bb\notin\Col(C)$;
\item $\nul(A) = \nul(A(i))$ if and only if $\bb\in\Col(C)$ and $a \neq \bx\trans C\bx$ for each $\bx$ with $C\bx = \bb$;  
\item $\nul(A) - 1 = \nul(A(i))$ if and only if $\bb\in\Col(C)$ and $a = \bx\trans C\bx$ for some $\bx$ with $C\bx = \bb$.  
\end{enumerate}
\end{remark}

For the three cases in Remark~\ref{rem:threecases}, we say the index (or the corresponding vertex) $i$ is \emph{upper}, \emph{neutral}, and \emph{downer}, respectively.  Note that the deciding factor for an index to be neutral or downer is the entry $a$ in (\ref{eq:Apartition}); the remaining diagonal entries are free. These observations imply the well-known Proposition \ref{prop:unique_t_downer} below. We provide a proof for completeness.

When dimensions have been specified (or they are clear from context), $E_{i,j}$ is the matrix with $(i,j)$-entry equal to $1$ and all other entries equal to zero. Similarly, let $\be_i$ be the $i$-th standard basis vector of $\mathbb{R}^n$, i.e., the $i$-th column of the identity matrix $I_n$.  

\begin{proposition}\label{prop:unique_t_downer}
Let $A$ be an $n\times n$ symmetric matrix. Then $i$ is a neutral index if and only if there is a (unique) value $t\not=0$ such that $i$ is a downer index for $A + tE_{i,i}$.
\end{proposition}
\begin{proof}
Suppose, without loss of generality, that $1$ is a neutral index for $A$, and assume that $A$ has the form in \eqref{eq:Apartition}. By Remark~\ref{rem:threecases} there exists a vector $\bf x$ such that ${\bf b} = C{\bf x}$. Then
\[
\left[\begin{array}{c|c}
1 & -\bx\trans\\
\hline
\bzero & I
\end{array}\right]
(A + tE_{i,i})
\left[\begin{array}{c|c}
1 & \bzero\trans\\
\hline
-\bx & I
\end{array}\right]
=
\left[
\begin{array}{c|c}
(a + t)-{\bf x}\trans C{\bf x} & {\bf 0}\trans  \\ \hline
{\bf 0} & C
\end{array} \right]. 
\]
By Sylvester's law of inertia the nullity of the above matrix equals $\nul(A + tE_{i,i})$. Moreover, its nullity is greater than $\nul(C)$ if and only if $(a + t) = {\bf x}\trans C{\bf x}$. Thus there is a unique value $t\not=0$ such that $1$ as a downer index of $A + tE_{i,i}$. The converse follows a similar argument.
\end{proof}

The above proposition translates nicely in terms of nullity pairs.

\begin{corollary}
\label{cor:kk1}
Let $G$ be a graph and $i\in V(G)$.  Then $(k,k)$ is the $i$-nullity pair for some matrix in $\mptn(G)$ if and only if $(k+1, k)$ is the $i$-nullity pair for some matrix in $\mptn(G)$.  
\end{corollary}

We now turn our attention towards the SNIP. Let $A\in \msym$ and $i \in [n]$ be given. Observe that the equations  
\[
X = X\trans,\quad A\circ X = O,\quad I\circ X = O \quad \text{and}\quad (AX)(i,:] = O
\]
are equivalent to a system of linear equations in the entries of $X$, which means
that the set of solutions forms a subspace of $\msym$.
For matrix $A$ to have the $i$-SNIP is equivalent to this subspace being trivial.

\begin{example}\label{example:K13}
\label{ex:verification}
Label the vertices of the star $K_{1, 3}$ such that some pendent vertex
is labeled $i = 1$ and the center vertex is labeled $4$.
Let
\[
    A = \begin{bmatrix}
     0 & 0 & 0 & 1 \\
     0 & 0 & 0 & 1 \\
     0 & 0 & 0 & 1 \\
     1 & 1 & 1 & 0
    \end{bmatrix}
    \in\mptn(K_{1,3}).
\]
Let $X\in \msym[4]$ and assume $A\circ X = I\circ X = O$. Then 
\[
    X = \begin{bmatrix}
     0 & x & y & 0 \\
     x & 0 & z & 0 \\
     y & z & 0 & 0 \\
     0 & 0 & 0 & 0
    \end{bmatrix},
\]
for some real $x,y$ and $z$. By direct computation
\[
    AX = \begin{bmatrix}
     0 & 0 & 0 & 0 \\
     0 & 0 & 0 & 0 \\
     0 & 0 & 0 & 0 \\
     x+y & x+z & y+z & 0
    \end{bmatrix}.
\] 
Now assume that a solution to $(AX)(1,:] = O$ is required.
Writing down a column vector of coefficients for each
of the $12$ entries, solving this equation is equivalent to solving the following system of twelve linear equations in $x,y$ and $z$:
\[
    \begin{bmatrix} x & y & z \end{bmatrix}
    \begin{bmatrix}
     0 & 0 & 0 & 0 & 0 & 0 & 0 & 0 & 1 & 1 & 0 & 0 \\
     0 & 0 & 0 & 0 & 0 & 0 & 0 & 0 & 1 & 0 & 1 & 0 \\
     0 & 0 & 0 & 0 & 0 & 0 & 0 & 0 & 0 & 1 & 1 & 0 \\
    \end{bmatrix} = 
    \begin{bmatrix} 0 & 0 & 0 \end{bmatrix}.
\]
Let $\Psi$ be the coefficient matrix in the above matrix equation.  
Since $\Psi$ has full row-rank, $X = O$ is the only solution, and $A$ has the $1$-SNIP.
\end{example}

The observant reader may notice that the arguments in Example~\ref{ex:verification} can be successfully applied to any matrix in $\mptn(K_{1,3})$ with $i = 1$. Indeed, columns $10, 11$ and $12$ of the resulting coefficient matrix are always independent.


The process used in Example~\ref{ex:verification} to check if a matrix has the SNIP can be generalized to any matrix
by constructing a coefficient matrix in a similar manner to $\Psi$. Such coefficient matrices have been used to study other strong properties and are usually referred to as \emph{verification matrices}. While verification matrices are a useful tool, they are not the focus of this paper. See \cite{IEPG2} for more information.




\section{SNIP and SAP}
\label{sec:snipsap}



The SNIP has a very similar definition to that of the strong Arnold property. 
A symmetric matrix $A$ is said to have the \emph{strong Arnold property} (SAP) if $X = O$ is the only symmetric matrix that satisfies the equations $A\circ X = O$, $I\circ X = O$ and $AX = O$.

\begin{remark}
\label{rem:snipsap}
If a matrix $A$ has the $i$-SNIP for some $i$, then $A$ has the SAP.  
\end{remark}

Motivated by this remark, we now investigate connections between the SNIP and the SAP.  
It is known that every nonsingular matrix $A$ has the SAP since $AX = O$ implies $X = O$.  The following proposition shows that an analogous result also holds for the SNIP, but by a different argument.

\begin{proposition}
\label{removedible}
Every nonsingular matrix has the $i$-SNIP for each index $i$.
\end{proposition}
\begin{proof}
Let $A$ be a symmetric nonsingular matrix and let $X$ by a symmetric matrix.  Suppose that $(AX)(i,:] = O$.  Then the columns of $AX$ all lie in $\vspan\{\be_i\}$. Thus, each column of $X$ lies in $\vspan\{A^{-1}\be_i\}$. Suppose $I\circ X = O$. Since $X$ is symmetric and has rank at most 1, $X = c(A^{-1}\be_i)\trans A^{-1}\be_i$ for some constant $c$.  However, such a matrix cannot have zero diagonal unless $c = 0$.  Therefore, $X=O$ and $A$ has the SNIP.
\end{proof}

It is also known that every symmetric matrix $A$ with $\nul(A) = 1$ has the SAP \cite{IEPGZF22}. However, this is not true for the SNIP, as we illustrate in the next example.  

\begin{example}
\label{ex:nul1}
Let 
\[
    A = \begin{bmatrix}
     0 & 1 & 1 \\
     1 & 0 & 0 \\ 
     1 & 0 & 0 \\
    \end{bmatrix} \text{ and }
    X = \begin{bmatrix}
     0 & 0 & 0 \\
     0 & 0 & 1 \\ 
     0 & 1 & 0 \\
    \end{bmatrix}.
\]
Then $A\circ X = O$, $I\circ X = O$, and $(AX)(1,:] = O$.  Therefore, $A$ has $\nul(A) = 1$, but $A$ does not have the $1$-SNIP. 
\end{example}

While $\nul(A) = 1$ does not guarantee the $i$-SNIP for any index $i$, Remark~\ref{rem:threecases} can be used to show that if $\nul(A) = 1$ there exists an index $i$ such that $A(i)$ is nonsingular.  

\begin{proposition}
\label{prop:pinvertible}
Let $A\in \msym$ and $i\in[n]$. If $A(i)$ is nonsingular, then $A$ has the $i$-SNIP.
\end{proposition}
\begin{proof}
Suppose $A(i)$ is nonsingular. Let $X$ be a symmetric matrix that satisfies $A\circ X = O$, $I\circ X = O$, and $(AX)(i,:] = O$.  We may assume $i = 1$ and that $A$ is of the form in \eqref{eq:Apartition}, where $C = A(i)$ is nonsingular.  Then $(AX)(i,:] = O$ is equivalent to   
\[
    \left[\begin{array}{c|c}
     \bb & C
    \end{array}\right]
    \left[\begin{array}{c|c}
    0 & \by\trans \\
    \hline
    \by & Y
    \end{array}\right] 
    = 
    \left[\begin{array}{c|c}
    C\by & \bb\by\trans + CY
    \end{array}\right]
    = O,
    \quad \text{where} \quad
    X = 
    \left[\begin{array}{c|c}
    0 & \by\trans \\
    \hline
    \by & Y
    \end{array}\right].
\]
Note that the $(i,i)$-entry of $X$ is $0$ since $I\circ X = O$.  Observe that $C\by = \bzero$, which implies $\by = \bzero$.  Substituting into the equation above gives $CY = O$ and hence $Y = O$.  Therefore, $X = O$ and $A$ has the $i$-SNIP.
\end{proof}

For symmetric matrices $A$ and $B$, it is known that $A\oplus B$ has the SAP if and only if one of $A$ and $B$ has the SAP while the other is nonsingular.  As the next result illustrates, the SNIP possesses a similar property. 

\begin{proposition}
\label{prop:directsum}
Let $A\in\msym$, $i\in[n]$, and $B\in\msym[m]$.  Then the direct sum $A\oplus B$ has the $i$-SNIP if and only if $A$ has the $i$-SNIP and $B$ is nonsingular.
\end{proposition}
\begin{proof}
Let $X_A\in\msym$, $X_B\in\msym[m]$ and
\[
    X = \begin{bmatrix}
    X_A & Y\trans \\
    Y & X_B
    \end{bmatrix}.
\]
Suppose that $A\circ X_A = I_n \circ X_A = O_n$ and $B\circ X_B = I_m\circ X_B = O_m$. Note that these equations hold if and only if $(A\oplus B)\circ X = I\circ X = O$. Also suppose that $(AX_A)(i,:] = O$, $(AY)(i,:] = O$, $BY = O$ and $BX_B = O$. Similarly, this second system of equations holds if and only if $((A\oplus B)X)(i,:] = O$ since
\[
    (A\oplus B)X = \begin{bmatrix}
     AX_A & AY\trans \\
     BY & BX_B
    \end{bmatrix}.
\]

Begin by assuming $A$ has the $i$-SNIP and $B$ is nonsingular.  Then $(AX_A)(i,:] = O$ implies $X_A = O$ since $A$ has the $i$-SNIP. Furthermore, $BY = O$ and $BX_B = O$ imply $Y = O$ and $X_B = O$ since $B$ is nonsingular. Therefore, $X = O$ and so $A\oplus B$ has the $i$-SNIP.  

Now assume $A\oplus B$ has the $i$-SNIP.  We first observe that there exists a nonzero vector $\ba$ such that $A\ba\in\vspan(\{\be_i\})$; namely, if $A$ is singular, then choose $\ba$ as a nonzero vector in $\ker(A)$, and  if $A$ is nonsingular, then choose $\ba$ as the vector such that $A\ba = \be_i$. Thus, $B$ is nonsingular, for otherwise $X_A = O$, $X_B = O$ and $Y= \bb\ba\trans$ would satisfy $(A\oplus B)\circ X = I\circ X = O$ and $((A\oplus B)X)(i,:] = O$, contradicting our assumption that $A\oplus B$ has the $i$-SNIP.
It follows from the equivalencies established in the first paragraph that $A$ has the $i$-SNIP.
\end{proof}

For the remainder of this section, we provide necessary and sufficient conditions for the SNIP in terms of the SAP.  

\begin{lemma}
\label{lem:snip2sap}
Let $A$ be an $n\times n$ symmetric matrix and $i\in[n]$. If $A$ has the $i$-SNIP, then both $A$ and $A(i)$ have the SAP.
\end{lemma}
\begin{proof}
By definition, if $A$ has the $i$-SNIP, then $A$ has the SAP.

To see that $A(i)$ has the SAP, suppose $Y$ is a symmetric matrix such that $A(i)\circ Y = O$, $I\circ Y = O$, and $A(i)Y = O$. 
Without loss of generality, suppose $i = 1$ and let $X = \begin{bmatrix} 0 \end{bmatrix}\oplus Y$.  Then $A\circ X = O$ and $I\circ X = O$.  We may assume $A$ has the form in \eqref{eq:Apartition} and compute  
\[
    AX = \begin{bmatrix}
     0 & \bb\trans Y \\
     \bzero & CY
    \end{bmatrix}
    =
    \begin{bmatrix}
     0 & \bb\trans Y \\
     \bzero & A(i)Y
    \end{bmatrix}
    =
    \begin{bmatrix}
     0 & \bb\trans Y \\
     \bzero & O
    \end{bmatrix}.
\]
Since $A$ has the $i$-SNIP, $X = O$. Therefore, $Y = O$ and $A(i)$ has the SAP.
\end{proof}

The converse of Lemma~\ref{lem:snip2sap} is not true in general.  

\begin{example}
\label{ex:doubleSAPnoSNIP}
By direct computation, for 
\[
A= \begin{bmatrix} 
1 & 0 \\ 
0 & 0
\end{bmatrix},
\]
both the matrix $A$ and $A(1)$ have the SAP, but $A$ does not have the $1$-SNIP.

\end{example}

A simple but useful observation is the following. Since $AX$ and $(A + tE_{i,i})X$ only differ in their $i$-th row, $(AX)(i,:] = O$ if and only if $((A + tE_{i,i})X)(i,:] = O$. Thus we have the following useful lemma.

\begin{lemma}
\label{lem:diagperturb}
Let $A$ be a symmetric matrix.  Then for each $t\in\mathbb{R}$, $A$ has the $i$-SNIP if and only if $A + tE_{i,i}$ has the $i$-SNIP.
\end{lemma}
 
Recall that an index can be either downer, neutral, or upper as defined in Section~\ref{sec:preliminaries}. 
In the following, we provide equivalent conditions for the $i$-SNIP under the assumptions that $i$ is a downer, neutral, or upper index.  The following result shows that counterexamples for the converse of Lemma~\ref{lem:snip2sap} (as in Example \ref{ex:doubleSAPnoSNIP}) can only happen when $i$ is a neutral index.

\begin{theorem}
\label{thm:snipchar}
Let $A\in\msym$ and $i\in [n]$.  Then the following characterization holds.
\begin{enumerate}[label={{\rm(\alph*)}}]
\item If  $i$ is a downer index, then $A$ has the $i$-SNIP if and only if $A$ has the SAP.
\item If $i$ is a neutral index and $t$ is the unique value such that $i$ is a downer index of $A + tE_{i,i}$, then  $A$ has the $i$-SNIP if and only if $A + tE_{i,i}$ has the SAP.
\item If $i$ is an upper index, then $A$ has the $i$-SNIP if and only if $A(i)$ has the SAP.
\end{enumerate}
\end{theorem}
\begin{proof}
By Lemma~\ref{lem:snip2sap}, if $A$ has the $i$-SNIP, then $A$ and $A(i)$ have the SAP.  Moreover, Lemma~\ref{lem:diagperturb} implies that if $A$ has the $i$-SNIP, then $A + tE_{i,i}$ has $i$-SNIP and consequently the SAP as well.  Thus, we have obtained the forward direction in all three cases.

We now prove the backward direction for each of the three cases. Let $X\in\msym$ such that $A\circ X = O$, $I\circ X = O$, and $(AX)(i,:] = O$.

Begin by assuming $i$ is a downer index and that $A$ has the SAP. By Remark~\ref{rem:threecases} the $i$-th row of $A$ is a linear combination of the remaining rows of $A$.  Since $(AX)(i,:] = O$ the columns of $X$ are orthogonal to the rows of $A(i,:]$ and hence orthogonal to the $i$-th row of $A$. This implies $AX = O$. Since $A$ has the SAP, $X = O$ and so $A$ has the $i$-SNIP.

Now assume $i$ is a neutral index and $A + tE_{i,i}$ has the SAP, where $t$ is the unique value such that $i$ is a downer index of $A + tE_{i,i}$. By the previous case, $A + tE_{i,i}$ has the $i$-SNIP.  Consequently, $A$ has the $i$-SNIP by Lemma~\ref{lem:diagperturb}.

Finally, assume $i$ is an upper index and $A(i)$ has the SAP.  Without loss of generality, we may assume $i = 1$ and that $A$ is of the form in \eqref{eq:Apartition}.  By Remark~\ref{rem:threecases}, $\bb$ is not a linear combination of the columns of $C$.  This implies that every vector $\bx$ satisfying $\begin{bmatrix} \bb & C \end{bmatrix}\bx = \bzero$ has its $i$-th entry equal to zero. Since 
\[
    (AX)(i,:] = \begin{bmatrix} \bb & C \end{bmatrix} X = O,
\]
the $i$-th row of $X$ is zero and $X = \begin{bmatrix} 0 \end{bmatrix}\oplus X(i)$.  Consequently, $ A(i)X(i) = O$.
Now $X(i)$ is a symmetric matrix such that $A(i)\circ X(i) = O$, $I\circ X(i) = O$, and $A(i)X(i) = O$.  Since $A(i)$ has the SAP, $X(i) = O$ and $X = \begin{bmatrix} 0 \end{bmatrix}\oplus X(i) = O$. Thus $A$ has the $i$-SNIP.
\end{proof}

We conclude this section by observing that Lemma~\ref{lem:diagperturb} and Proposition~\ref{prop:unique_t_downer} give a version of Corollary~\ref{cor:kk1} that preserves the SNIP.
\begin{corollary}
\label{cor:kk1withSNIP}
Let $G$ be a graph and $i\in V(G)$. Then $(G,i)$ allows the nullity pair  $(k,k)$ with the SNIP if and only if $(G,i)$ allows the nullity pair $(k+1, k)$ with the SNIP.
\end{corollary}

\section{Nullity pairs for families of rooted graphs}
\label{sec:inpp}

In this section we study the realizable nullity pairs for various families of rooted graphs, including the complete graphs $K_n$, the cycle graphs $C_n$, and the path graphs $P_n$. It is of particular interest when there exists a realization with the SNIP. For a vertex transitive graph $G$, e.g., $K_n$ and $C_n$, we write $(G,i)$ without specifying the vertex $i$. Similarly, when the context is clear, we write $(G,\mbox{leaf})$ if the root is a leaf, i.e., the root is a vertex of degree 1. We start with two preliminary results that guarantee certain nullity pairs are realizable with the SNIP for almost every rooted graph.  

\begin{proposition}
\label{prop:010}
Let $(G,i)$ be a rooted graph. Then $(G,i)$ allows the nullity pairs $(0,0)$ and $(1,0)$ with the SNIP. 
\end{proposition}
\begin{proof}
Let $A\in\mptn(G)$. For $\lambda$ large enough $A + \lambda I$ has $i$-nullity pair $(0,0)$. By Corollary~\ref{cor:kk1} and Proposition~\ref{prop:pinvertible}, both $(0,0)$ and $(1,0)$ are realizable for $(G,i)$ with the SNIP.
\end{proof}

\begin{proposition}
\label{prop:01}
Let $(G,i)$ be a rooted graph such that  $i$ is not an isolated vertex. Then $(G,i)$ allows the nullity pair $(0,1)$ with the SNIP.
\end{proposition}
\begin{proof}
Without loss of generality, assume $i = 1$. We may write $G - i$ as the disjoint union of $H_1$ and $H_2$ (possibly empty)  such that $i$ is adjacent to a vertex in $H_1$ and $H_1$ is connected.  Let $B_1$ be the Laplacian matrix of $H_1$ and, if $H_2$ is nonempty, let $B_2$ be a nonsingular matrix in $\mptn(H_2)$. Then there exists a matrix $A\in \mptn(G)$ of the form  
\[
    A = \begin{bmatrix}
    a & \bb_1\trans & \bb_2\trans \\
    \bb_1 & B_1 & O \\
    \bb_2 & O & B_2
    \end{bmatrix},
\]
where $\bb_1$ is the appropriately chosen $(0,1)$-vector. As $H_1$ is connected, the column space of $B_1$
consists of all vectors whose entries sum to $0$.

Since $i$ is adjacent to some vertex in $H_1$,
$\bb_1$ is not in the column space of $B_1$. 
By construction,  $\nul(A(i)) = 1$.  Since $\bb_1\notin\Col(B_1)$, Remark~\ref{rem:threecases} implies $\nul(A) = 0$. Proposition~\ref{removedible} implies $A$ has the $i$-SNIP.
\end{proof}



Since we are primarily interested in connected graphs and every connected graph other than $K_1$ allows the nullity pairs $(0,0), (1,0)$ and $(0,1)$ with the SNIP, we shall refer to these nullity pairs as \emph{trivial}.

All of the proofs in the remainder of this section rely on known values of the maximum nullity for specific graphs. We refer the reader to \cite[Chapter 2]{IEPGZF22} for a review of this information.


\begin{proposition}
\label{prop:Kn}
Let $n\geq 2$ and $k,\ell\geq 0$ be integers. Then $(K_n,i)$ allows the nullity pair $(k,\ell)$ with the SNIP if and only if $|k - \ell| \leq 1$ and $\ell \leq n - 2$, or $(k,\ell) = (0,1)$. Moreover, every matrix $A\in\mptn(K_n)$ has the $i$-SNIP. 
\end{proposition}
\begin{proof}
Without loss of generality, we may assume $i = 1$. By Propositions \ref{prop:010} and \ref{prop:01}, $(K_n,i)$ allows $(0,0), (1,0)$ and $(0,1)$ with the SNIP. Since $M(K_2) = 1$, the only remaining possibility for $n=2$ is the nullity pair $(1,1)$.  However, $(1,1)$ not realizable for any root since Corollary~\ref{cor:kk1} would imply $(2,1)$ is realizable, violating $M(K_2) = 1$.  This proves the case of $n = 2$, so assume $n\geq 3$.

Suppose that $(K_n,i)$ allows the nullity pair $(k,\ell)$ with the SNIP. As always, $|k-\ell|\leq 1$. Since $M(K_n - i) = n-2$ it follows that $\ell \leq n-2$.

Observe that every matrix $A\in \mptn(K_n)$ has the $i$-SNIP since $A\circ X = O$ and $I\circ X = O$ imply $X = O$. Thus, to establish the remaining direction, it suffices to exhibit matrices with the desired $i$-nullity pairs. For $1 \leq m \leq n - 2$ let $B_{m,n}\in\mptn(K_{n-1})$ be the matrix obtained from the adjacency matrix of $K_{n-1}$ by adding $1$ to the first $m+1$ diagonal entries. Observe that $\nul(B_{m,n}) = m$ and $B_{m,n} \bone$ is entrywise positive. Let 
\[
\widehat{B}_{m,n} = 
\begin{bmatrix} 
\bone\trans B_{m,n}\bone & (B_{m,n}\bone)\trans \\
B_{m,n}\bone & B_{m,n}
\end{bmatrix}.
\] 
By construction, the matrix $\widehat{B}_{m,n} \in \mptn(K_n)$ and $\nul(\widehat{B}_{m,n}) = m + 1$. 

Now suppose $|k - \ell| \leq 1$ and $1\leq\ell \leq n - 2$. Then $k \in \{ \ell -1, \ell, \ell + 1\}$. Since $\widehat{B}_{\ell,n}$
has $i$-nullity pair $(\ell+1,\ell)$,  Corollary~\ref{cor:kk1} implies $(K_n,i)$ allows $(\ell+1,\ell)$ and $(\ell,\ell)$. It remains to show that $(\ell - 1,\ell)$ is realizable. Observe that
\[
\by = \begin{bmatrix} 1 \\ -\bone \end{bmatrix}\in\ker(\widehat{B}_{\ell-1,n-1})
\]
and that $\by$ does not contain any zero entries. Thus
\[
\begin{bmatrix}
0 & \by\trans \\
\by & \widehat{B}_{\ell-1,n-1}
\end{bmatrix}
\]
is in $\mptn(K_n)$ and has $i$-nullity pair $(\ell - 1,\ell)$, as required.
\end{proof}


\begin{proposition}
\label{prop:Cn}
Let $n\geq 3$. Then the nontrivial nullity pairs allowed by $(C_n, i)$ with the SNIP are precisely $(1,1)$ and $(2,1)$. Moreover, every matrix in $\mptn(C_n)$ has the $i$-SNIP.
\end{proposition}
\begin{proof}
Observe that $C_n - i = P_{n-1}$. Since $M(C_n) = 2$ and $M(P_{n-1}) = 1$, there exists a matrix in $\mptn(C_n)$ with $i$-nullity pair $(2,1)$. Thus, $(C_n,i)$ allows $(2,1)$, and by Corollary~\ref{cor:kk1}, $(C_n,i)$ allows $(1,1)$. Note that the only possible remaining nullity pairs are trivial. Since every matrix in $\mptn(C_n)$ and $\mptn(P_{n-1})$ has the SAP \cite[Theorem~2.6]{Zsap}, every matrix in $\mptn(C_n)$ has the $i$-SNIP by Theorem~\ref{thm:snipchar}.
\end{proof}


\begin{proposition}\label{prop:paths}
Let $n\geq 2$. Then $(P_n,\mbox{leaf})$ allows the nullity pair $(k,\ell)$ if and only if $(k,\ell)$ is trivial. For $i\in V(P_n)$ that is not a leaf, the only nontrivial nullity pair allowed by $(P_n,i)$ is $(1,2)$. Moreover, a matrix in $\mptn(P_n)$ has the $i$-SNIP for any index $i$ if and only if its $i$-nullity pair is trivial.
\end{proposition}
\begin{proof}
Since $M(P_n) = 1$, we only need to consider nullity pairs $(k,\ell)$ with $k\leq 1$ and $|k - \ell| \leq 1$. We can eliminate $(1,1)$ for any root since Corollary~\ref{cor:kk1} would imply $(2,1)$ is realizable. This leaves $(1,2)$ as the only possible nontrivial nullity pair for any root. 

Suppose that $i\in V(P_n)$ is a leaf. Then $M(P_n - i) = M(P_{n-1}) = 1$ and so $(P_n,\mbox{leaf})$ does not allow $(1,2)$.

Suppose that $i\in V(P_n)$ is not a leaf. Then $M(P_n - i) = 2$ and so $(P_n,i)$ allows $(1,2)$. Moreover, every matrix in $\mptn(P_n - i)$ with nullity $2$ has the form $B = B_1\oplus B_2$ such that $\nul(B_1) = \nul(B_2) = 1$. By \cite[Lemma~3.1]{BFH05}, $B$ does not have the SAP and each $A\in\mptn(P_n)$ with $A(i) = B$ does not have the $i$-SNIP by Theorem~\ref{thm:snipchar}.

Moreover, by Propositions \ref{prop:010} and \ref{prop:01}, a matrix in $\mptn(P_n)$ has the $i$-SNIP for any index $i$ if and only if its $i$-nullity pair is trivial.
\end{proof}


A vertex in the star graph $K_{1,n}$ is either a leaf or the center vertex. We write $(K_{1,n},\mbox{center})$ if the center vertex is the root. In the next section, Theorem \ref{thm:trees} characterizes when a rooted tree allows a nullity pair with the SNIP. As such, Proposition \ref{prop:stars} only determines the nullity pairs realizable by rooted stars. 

\begin{proposition}
\label{prop:stars}
Let $n\geq 3$. The nontrivial nullity pairs allowed by $(K_{1,n},\mbox{center})$ are precisely the pairs $(\ell - 1,\ell)$ with $2\leq \ell \leq n$. The nontrivial nullity pairs allowed by $(K_{1,n},\mbox{leaf})$ are precisely the pairs  $(\ell,\ell)$ and $(\ell+1,\ell)$ with $1\leq \ell \leq n-2$.
\end{proposition}
\begin{proof}
Let $A = [a_{ij}]\in \mptn(K_{1,n})$ and assume that $1$ is the central vertex. Let $t$ be the number of zero diagonal entries of $A(1)$. Observe that if $t>0$, then $\nul(A)=t-1$ (and if $t=0$, then $\nul(A) \in \{0,1\}$). 

Since $\nul(A(1)) = t$, the nontrivial nullity pairs allowed by $(K_{1,n},\mbox{center})$ are precisely the pairs $(\ell - 1,\ell)$ with $2\leq \ell \leq n$.

Let $i\in V(K_{1,n})$ be a leaf and suppose $2 \leq t \leq n$. Observe that $\nul(A(i)) = t-1$ if $a_{ii}\not=0$ and $\nul(A(i)) = t-2$ if $a_{ii} = 0$. 
Thus the nontrivial nullity pairs allowed by $(K_{1,n},\mbox{leaf})$ are precisely the pairs  $(\ell,\ell)$ and $(\ell+1,\ell)$ with $1\leq \ell \leq n-2$.
\end{proof}

\section{Minimal minors}
\label{sec:minors}

Having introduced the SNIP and established the necessary tools, we are ready to showcase the potential of the Minor Monotonicity Theorem (Theorem~\ref{thm:minmono}). In particular, we determine the set of minimal minors for nullity pairs $(k,\ell)$ with the SNIP when $k\leq 1$ or $\ell\leq 1$. Observe that it suffices to only consider the connected case when classifying minimal minors. Indeed, by Proposition \ref{prop:directsum}, for any disconnected rooted graph $(G,i)$, the component $H$ that contains $i$ give a connected minor $(H,i)$ that allows $(k,\ell)$ with the SNIP.  Therefore, every minimal minor for a given nullity pair is connected.  Moreover,
according to the celebrated Robinson--Seymour--Thomas Theorem (graph minor theorem) \cite{RS2010}, the set of minimal minors is finite for each $(k,\ell)$.  

\begin{remark}\label{rem:minfortrivial}
By Proposition~\ref{prop:010}, the only minimal minor for nullity pairs $(0,0)$ and $(1,0)$ is $(K_1,i)$. Similarly, Proposition \ref{prop:01} implies that the only minimal minor for the nullity pair $(0,1)$ is $(K_2,i)$.
\end{remark}

\subsection{Nullity pairs \texorpdfstring{$(1,1)$}{(1,1)}, \texorpdfstring{$(2,1)$}{(2,1)}}

The minimal minors for $(1,1)$ and $(2,1)$ are identical by Corollary~\ref{cor:kk1} and Lemma~\ref{lem:diagperturb}. That being said, they are a bit more complicated than the trivial nullities established in Remark~\ref{rem:minfortrivial}. 


A vertex with degree at least $3$ is called a \emph{high-degree} vertex.  A \emph{generalized star} is a tree with exactly one high-degree vertex, and the unique high-degree vertex is called its center. The following lemma can be quickly proved using the Parter--Wiener theorem, see \cite{P1960,W1984,JDS2003}. 

\begin{lemma}
\label{lem:gstar}
Let $G$ be a generalized star.
Then $(G,\mbox{center})$ does not allow the nullity pair $(2,1)$.  
\end{lemma}
\begin{proof}
Suppose $A\in\mptn(G)$ satisfies $\nul(A) = 2$ and let $i$ be the center vertex of $G$. By the Parter--Wiener theorem, the vertex $i$ is the Parter vertex, meaning $\nul(A(i)) = 3$.
\end{proof}

\begin{theorem}
\label{thm:minl21}
Let $(G,i)$ be a connected rooted graph.  Then the following are equivalent.  
\begin{enumerate}[label={{\rm(\alph*)}}]
\item $(G,i)$ allows the nullity pair $(2,1)$ with the SNIP.
\item $(G,i)$ allows the nullity pair $(2,1)$.
\item $(G,i)$ contains $(K_3,i)$ or $(K_{1,3},\mbox{leaf})$ as a minor.
\item $G$ is not a path, and $(G,i)$ is not a generalized star with the root $i$ at its center.
\end{enumerate}
\end{theorem}
\begin{proof}
By Proposition~\ref{prop:Kn} and Example~\ref{example:K13}, both $(K_3,i)$ and $(K_{1,3},\mbox{leaf})$ allow the nullity pair $(2,1)$ with the SNIP.  By the Minor Monotonicity Theorem, 
every rooted graph that contains $(K_3,i)$ or $(K_{1,3},\mbox{leaf})$ as a minor allows $(2,1)$ with the SNIP. This proves (c) implies (a).  Also, note (a) implies (b) by definition.

Now suppose $(G,i)$ does not contain $(K_3, i)$ or $(K_{1,3}, \mbox{leaf})$ as a minor.  
Observe that $G$ is a tree since $(G,i)$ is connected and does not contain $(K_3,i)$ as a minor. Assume that $G$ has a high-degree vertex $j\not=i$. Contract edges on the path connecting $i$ to $j$ until $i$ is adjacent to $j$. Observe that the resulting rooted graph contains $(K_{1,3}, \mbox{leaf})$ as an induced rooted subgraph. Therefore, $(K_{1,3}, \mbox{leaf})$ is a minor of $(G,i)$. This is a contradiction and so $G$ can have at most one high-degree vertex. If $G$ has one high-degree vertex, then $G$ is a generalized star and $i$ has to be its center.  If $G$ has no high-degree vertex, then $G$ is a path. Furthermore, both a rooted path (rooted at any vertex) or a generalized star rooted at its center does not contain $(K_3,i)$ or $(K_{1,3}, \mbox{leaf})$ as a minor.  Thus, (c) and (d) are equivalent.

To complete the proof, we show (b) implies (d) via contraposition.  If $G$ is a path, then $(G,i)$ does not allow $(2,1)$ for any root $i$ since the maximum nullity of a path is $1$.  If $G$ is a generalized star with $i$ its center, then $(G,i)$ does not allow $(2,1)$ as well by Lemma~\ref{lem:gstar}.
\end{proof}

\subsection{Nullity pair \texorpdfstring{$(1,2)$}{(1,2)}}

As we shall soon see, the minimal minors for the nullity pair $(1,2)$ with the SNIP are the two rooted graphs shown in Figure~\ref{fig:paws211}.  
\begin{figure}[h]
\begin{center}
\begin{tikzpicture}
\node[label={above:$i$}] (1) at (-1,0) {};
\node (2) at (0,0) {};
\node (3) at (30:1) {};
\node (4) at (-30:1) {};
\draw (1) -- (2) -- (3) -- (4) -- (2);
\node[rectangle,draw=none] at (0,-1) {$\Paw$};
\end{tikzpicture}
\hfil
\begin{tikzpicture}
\node[label={above:$i$}] (1) at (-2,0) {};
\node (2) at (-1,0) {};
\node (3) at (0,0) {};
\node (4) at (30:1) {};
\node (5) at (-30:1) {};
\draw (1) -- (2) -- (3);
\draw (4) -- (3) -- (5);
\node[rectangle,draw=none] at (-0.5,-1) {$S(2,1,1)$};
\end{tikzpicture}
\end{center}
\caption{The minimal minors for the nullity pair $(1,2)$ with the SNIP.}
\label{fig:paws211}
\end{figure}
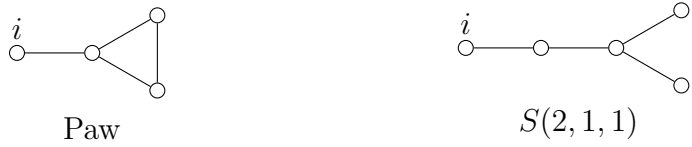

If unspecified we shall always assume the roots of $(\Paw,i)$ and $(S(2,1,1),i)$ are as indicated in Figure~\ref{fig:paws211}. The next example verifies that these rooted graphs allow the nullity pair $(1,2)$ with the SNIP.

\begin{example}
\label{ex:pair12}
Let  
\[A_1 = 
\begin{bmatrix}
 0 & 1 & 0 & 0 \\
 1 & 1 & 1 & 1 \\
 0 & 1 & 1 & 1 \\
 0 & 1 & 1 & 1 \\
\end{bmatrix} \quad\text{ and }\quad
A_2 = 
\begin{bmatrix}
 0 & 1 & 0 & 0 & 0 \\
 1 & 0 & 0 & 0 & 1 \\
 0 & 0 & 0 & 0 & 1 \\
 0 & 0 & 0 & 0 & 1 \\
 0 & 1 & 1 & 1 & 0 \\
\end{bmatrix}.\]
Observe that $A_1 \in \mptn(\Paw)$ and $A_2\in\mptn(S(2,1,1))$.  By direct computation, both $A_1$ and $A_2$ have $1$-nullity pair $(1,2)$.  Also, $A_1(1)$ and $A_2(1)$ has the SAP, so $A_1$ and $A_2$ have the $1$-SNIP by Theorem~\ref{thm:snipchar}. Therefore, both $(\Paw, 1)$ and $(S(2,1,1), 1)$ allow the nullity pair $(1,2)$ with the SNIP.
\end{example}

We will show that the absence of $(\Paw,i)$ and $(S(2,1,1),i)$ minors results in a yam graph (see Figure~\ref{fig:yam}).  

\begin{definition}
A \emph{yam graph} is a connected rooted graph $(G,i)$ such that 
\begin{itemize}
\item each component of $G - i$ is either a generalized star or a path,
\item $i$ is only adjacent to the center of each component of $G - i$ that is a generalized star.
\end{itemize}
\end{definition}

\begin{figure}[h]
\begin{center}
\begin{tikzpicture}
\node[label={right:root}, fill] (root) at (0,0) {};

\node (l11) at (-1,1) {};
\node (l12) at (-1.5,2) {};
\node (l13) at (-3,1) {};
\node (l14) at (-3,0.5) {};
\draw (root) -- (l11);
\draw decorate[decoration={name=random steps}] {(l11) -- (l12)};
\draw decorate[decoration={name=random steps}] {(l11) -- (l13)};
\draw decorate[decoration={name=random steps}] {(l11) -- (l14)};

\node (l21) at (-0.5,2) {};
\node (l22) at (-1.5,3) {};
\node (l23) at (-1,3.5) {};
\node (l24) at (0.5,3) {};
\node (l25) at (2,3) {};
\draw (root) -- (l21);
\draw decorate[decoration={name=random steps}] {(l21) -- (l22)};
\draw decorate[decoration={name=random steps}] {(l21) -- (l23)};
\draw decorate[decoration={name=random steps}] {(l21) -- (l24)};
\draw decorate[decoration={name=random steps}] {(l21) -- (l25)};

\node (l31) at (1,1.5) {};
\node (l32) at (2,2) {};
\node (l33) at (4,1.5) {};
\node (l34) at (3,1) {};
\draw (root) -- (l31);
\draw decorate[decoration={name=random steps}] {(l31) -- (l32)};
\draw decorate[decoration={name=random steps}] {(l31) -- (l33)};
\draw decorate[decoration={name=random steps}] {(l31) -- (l34)};

\foreach \i in {1,2,3,4} {
    \pgfmathsetmacro{\x}{-2 + 0.5*(\i - 1)}
    \node (y1\i) at (\x,-1) {};
}
\draw (y11) -- (y12) -- (y13) -- (y14);
\draw (root) -- (y11);
\draw (root) -- (y12);
\draw (root) -- (y14);
\draw[dashed] (-1.25,-1) ellipse (1cm and 0.3cm); 

\foreach \i in {1,2,3,4,5,6} {
    \pgfmathsetmacro{\x}{0 + 0.5*(\i - 1)}
    \node (y2\i) at (\x,-2) {};
}
\draw (y21) -- (y22) -- (y23) -- (y24) -- (y25) -- (y26);
\draw (root) -- (y22);
\draw (root) -- (y23);
\draw (root) -- (y26);
\draw[dashed] (1.25,-2) ellipse (1.75cm and 0.3cm); 

\end{tikzpicture}
\end{center}
\caption{An example of yam graph.}
\label{fig:yam}
\end{figure}
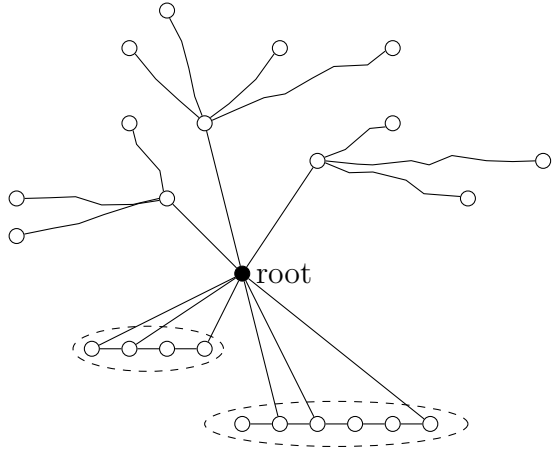

In order to characterize the minimal minors for the nullity pair $(1,2)$ we require the following two lemmas.

\begin{lemma}
\label{lem:gstarnul2}
Let $(G,i)$ be a generalized star with the root $i$ at its center.  If $B\in\mptn(G)$ has nullity at least $2$, then each vector in $\ker(B)$ has its $i$-th coordinate equal to $0$. 
\end{lemma}
\begin{proof}
Assume $B\in \mptn(G)$ has nullity at least 2. By the Parter--Wiener theorem, 
the vertex $i$ is upper. 
By Remark~\ref{rem:threecases}, the $i$-th column of $B$ is not a linear combination of the other columns. Let $\bx\in\ker(B)$. Since $B\bx$ is a linear combination of the columns of $B$, the $i$-th entry of $\bx$ is zero.
\end{proof}

\begin{lemma}
\label{lem:yam12snip}
A yam graph does not allow $(1,2)$ with the SNIP.  
\end{lemma}
\begin{proof}
Let $(G,i)$ be a yam graph.  Suppose $A\in\mptn(G)$ has the $i$-SNIP and $i$-nullity pair $(1,2)$. Without loss of generality, $i = 1$ and $A(i)$ is the direct sum of matrices $B_1, \ldots, B_h$, each of which corresponds to a component of $G - i$. So, $A$ has the form
\[
    A = \begin{bmatrix}
    a & \bb_1\trans & \cdots & \bb_h\trans \\
    \bb_1 & B_1 & ~ & ~ \\
    \vdots & ~ & \ddots & ~ \\
    \bb_h & ~ & ~ & B_h
    \end{bmatrix}.
\]
By Theorem~\ref{thm:snipchar}, $A(i)$ has the SAP with $\nul(A(i)) = 2$. By \cite[Lemma~3.1]{BFH05}, we may assume $\nul(B_1) = 2$ and $B_2,\ldots, B_h$ are nonsingular.  Since the maximum nullity of a path is $1$, the graph of $B_1$ is a generalized star; let $j$ denote its center.  By Lemma~\ref{lem:gstarnul2}, each vector $\bx$ in $\ker(B_1)$ has its $j$-th entry zero. Padding each such $\bx$ with zeros results in a null vector of $A$ since the only nonzero entry in $\bb_1$ is the $j$-th entry.  Therefore, $\nul(A) \geq 2$, which is a contradiction.
\end{proof}

We are now ready to state the main result of this section.

\begin{theorem}
\label{thm:minl12}
Let $(G,i)$ be a connected rooted graph.  Then the following are equivalent.  
\begin{enumerate}[label={{\rm(\alph*)}}]
\item $(G,i)$ allows the nullity pair $(1,2)$ with the SNIP.
\item $(G,i)$ contains $(\Paw,i)$ and $(S(2,1,1),i)$, as shown in Figure~\ref{fig:paws211}.
\item $G$ is not a yam graph.
\end{enumerate}
\end{theorem}
\begin{proof}
By Example~\ref{ex:pair12}, both $(\Paw, i)$ and $(S(2,1,1),i)$ are rooted graphs that allow the nullity pair $(1,2)$ with the SNIP.  Therefore, by the Minor Monotonicity Theorem, 
every connected $(G,i)$ that contains $(\Paw, i)$ and $(S(2,1,1),i)$ as a minor allows $(1,2)$ with the SNIP. Thus (b) implies (a). By Lemma~\ref{lem:yam12snip}, a yam graph does not allow $(1,2)$ with the SNIP, which shows (a) implies (c).

To complete the proof, we now show that (c) implies (b). 
Assume $(G,i)$ does not contain $(\Paw,i)$ or $(S(2,1,1),i)$ as a minor.  Observe that if $G$ contains a cycle, then $i$ lies on that cycle, otherwise $(\Paw,i)$ is a minor of $(G,i)$.  As a result, $G - i$ is a forest.  If a component of $G - i$ contains a high-degree vertex $j$ (with respect to $G - i$) such that $i$ is not adjacent to $j$, then $(G,i)$ contains $(S(2,1,1),i)$ as a minor, a contradiction. Suppose that a component $H$ of $G - i$ has two high-degree vertices $u$ and $v$ (with respect to $G - i$). Then, by the preceding argument, $u$ and $v$ are adjacent to $i$. Since $H$ is connected, there exists a path of length at least 2 from $i$ to $u$ or $v$. Thus, $(G,i)$ contains $(S(2,1,1),i)$ as a minor, a contradiction. Consequently, a component of $G - i$ is either a generalized star or a path; when it is a generalized star, $i$ is only adjacent to its center. So, by definition, $(G,i)$ is a yam graph. 
\end{proof}

With the help of the SNIP, we are able to characterize all connected rooted graphs that allow $(1,2)$.

\begin{proposition}
\label{prop:cutvtx}
Every rooted graph $(G,i)$, where $i$ is a cut-vertex, allows the nullity pair $(1,2)$.
\end{proposition}
\begin{proof}
Let $(G,i)$ be a rooted graph, where $i$ is a cut vertex. Then $G-i$ contains at least two components $H_1$ and $H_2$.  We may therefore write $G - i$ as the disjoint union of $H_1$, $H_2$ and $H_3$, where $H_3$ is possibly empty or disconnected.  Let $L_1$ and $L_2$ be the Laplacian matrices of $H_1$ and $H_2$, respectively. Let $E\in\mptn(H_3)$ be nonsingular. Without loss of generality $i = 1$ and  
\[
    A = \begin{bmatrix}
    a & \bb_1\trans & \bb_2\trans & \bb_3\trans \\
    \bb_1 & L_1 & ~ & ~ \\
    \bb_2 & ~ & L_2 & ~ \\
    \bb_3 & ~ & ~ & E
    \end{bmatrix},
\]
where $a = 1$, and $\bb_1, \bb_2$ and $\bb_3$ are the appropriately chosen $(0,1)$-vectors. Observe that $\nul(A(i)) = 2$.  Since $i$ is a cut vertex, $\inp{\bb_1}{\bone} \neq 0$.  Since $\ker(L_1) = \vspan(\{\bone\})$, $\bb_1$ is not in the column space of $L_1$.  Therefore, by Remark~\ref{rem:threecases}, $i$ is a upper index, and so $A$ has $i$-nullity pair $(1,2)$.
\end{proof}

\begin{lemma}
\label{lem:yam12}
Let $(G,i)$ be a yam graph. Then $(G,i)$ allows the nullity pair $(1,2)$ if and only if $G - i$ has two or more components.
\end{lemma}
\begin{proof}
Suppose $G - i$ has two or more components. Then $i$ is a cut-vertex and by Proposition~\ref{prop:cutvtx}, $(G,i)$ allows $(1,2)$.  

Suppose $G - i$ has at most one component. If $G-i$ is the empty graph, then $(G,i)$ does not allow $(1,2)$. So assume that $G-i$ has exactly one component. If the component is a path, then it does not allow $(1,2)$ since the maximum nullity of a path is $1$.  Otherwise the component is a generalized star. Let $A\in\mptn(G)$. Just as in the proof of Lemma~\ref{lem:yam12snip}, Lemma~\ref{lem:gstarnul2} implies that if $\nul(A(i)) = 2$, then $\nul(A) \geq 2$. Thus $(G,i)$ does not allow the nullity pair $(1,2)$.
\end{proof}

The following theorem is now immediate from Theorem~\ref{thm:minl12} and Lemma~\ref{lem:yam12}

\begin{theorem}
\label{thm:minl12nosnip}
Let $G$ be a connected graph with vertex $i$. Then $(G,i)$ allows the nullity pair $(1,2)$ if and only if $(G,i)$ is not a yam graph such that $G - i$ has at most one component.
\end{theorem}

\subsection{Rooted tree}
The parameter $\xi(G)$, introduced and studied in \cite{BFH05}, is defined as the maximum nullity among matrices in $\mptn(G)$ with the SAP. By \cite[Theorem 3.7]{BFH05}, $\xi(T) = 2$ for every tree $T$ that is not a path and \cite[Theorem 3.2]{BFH05} implies $\xi(T-i)\leq 2$. These facts, along with various results from this paper, allow us to characterize the nullity pairs that are allowed by rooted trees with the SNIP.

\begin{theorem}
\label{thm:trees}
Let $T$ be a tree on $n\geq 2$ vertices such that $T\not= P_n$ and let $i\in V(T)$. If $(T,i)$ allows $(k,\ell)$ with the SNIP, then $k\leq 1$ or $\ell \leq 1$. The rooted tree $(T,i)$ allows $(2,1)$ and $(1,1)$ with the SNIP if and only if $T$ has a high-degree vertex $v\not=i$. The rooted tree $(T,i)$ allows $(1,2)$ with the SNIP if and only if $T$ contains a high-degree vertex $v\not=i$ such that $v$ is not adjacent to $i$.
\end{theorem}
\begin{proof}
Suppose that $(T,i)$ allows $(k,\ell)$ with the SNIP. Recall Lemma~\ref{lem:snip2sap} states that if a symmetric matrix $A$ has the SNIP, then $A$ and $A(i)$ have the SAP. Thus $\xi(T) = 2$ and $\xi(T-i)\leq 2$ imply $k\leq 2$ and $\ell\leq 2$.
Further, $(T,i)$ does not allow $(2,2)$ with the SNIP as Corollary \ref{cor:kk1withSNIP} would imply $(T,i)$ allows $(3,2)$ with the SNIP. Thus, $k\leq 1$ or $\ell \leq 1$.

Suppose that $(T,i)$ allows $(2,1)$ and $(1,1)$ with the SNIP. Since $T$ is not a path, either $T$ contains a high-degree vertex $v\neq i$ or $T$ is a generalized star with central vertex $i$. Lemma~\ref{lem:gstar} implies that if $T$ is a generalized star with central vertex $i$, then $T$ does not allow the nullity pair $(2,1)$. To prove the converse, suppose that $T$ contains a high-degree vertex $v\neq i$. Then $(T,i)$ contains $(K_{1,3}, \mbox{leaf})$ as minor. Since $(K_{1,3}, \mbox{leaf})$ allows $(2,1)$ with the SNIP, Theorem~\ref{thm:minmono} 
implies $(T,i)$ allows the nullity pair $(2,1)$ with the SNIP.

By Theorem \ref{thm:minl12}, $(T,i)$ allows $(1,2)$ with the SNIP if and only if $(S(2,1,1),i)$ is a rooted minor of $(T,i)$. This occur if and only if $T$ contains a vertex $v\not=i$ of degree at least 3 such that $v$ is not adjacent to $i$.
\end{proof}

\section{Developing the SNIP and its consequences}
\label{sec:theory}
In Section~\ref{sec:minors} we saw the power of the Minor Monotonicity Theorem. 
Here we provide the theoretical background of this theorem and other applications of the SNIP. 
Let $G$ be a graph with vertex $i$ and let $A\in\mptn(G)$. We begin by considering two types of perturbations of $A$ obtained by group actions. 

Let $\mptncl(G)$ be the \emph{topological closure} of $\mptn(G)$, i.e., all symmetric matrices whose $i,j$-entry is zero whenever $\{i,j\}\notin E(G)$ and $i\neq j$.
Note that the entries of a matrix in $\mptncl(G)$ that correspond to an edge are allowed to be zero. The set $\mptncl(G)$ is a subspace of the vector space $\msym$ and, in particular, is an additive group. Observe that any perturbation of $A\in\mptn(G)$ that takes the form
\[
    A \mapsto A + B
\]
for $B\in\mptncl(G)$ and $B$ sufficiently small preserves the pattern of $A$, i.e., $A + B \in \mptn(G)$.

We also consider the subgroup of the general linear group given by: 
\[
    \GL{i} := \{Q\in\mat : \det(Q) \neq 0,\ Q[i,i) = \bzero\trans\}.
\]
For example, when $i = 1$, the matrices in  $\GL{i}$ are of the form  
\[
    Q = \begin{bmatrix}
     q_{1,1} & \bzero\trans \\
     Q(1,1] & Q(1)
    \end{bmatrix}
\]
where $q_{1,1}\neq 0$ and $Q(1)$ is nonsingular.  By direct computation, we have $(Q\trans AQ)(i) = Q(i)\trans A(i)Q(i)$ for $Q\in \GL{i}$.  Thus, any perturbation of $A\in\msym$ that takes the form
\[
    A \mapsto Q\trans AQ
\]
for $Q \in \GL{i}$ preserves the $i$-nullity pair of $A$.

\begin{figure}[h]
\begin{center}
\begin{tikzpicture}
\clip (-4.5,-2) rectangle (4.5,2);
\draw (-3,0) -- (3,0);
\draw ([shift={(-40:5)}]-5,0) arc (-40:40:5);
\node[rectangle, fill=white, draw=none, text=black] at (-3,0) {same pattern}; 
\node[rectangle, fill=white, draw=none, text=black] at ([shift={(-20:5)}]-5,0) {same $i$-nullity pair};

\coordinate (a) at (0,0);
\node[label={225:$A$}, fill] (A) at (a) {};

\node[rectangle, fill=white, draw=none, text=black] at (3,0) {$A + B$}; 
\node[rectangle, fill=white, draw=none, text=black] at ([shift={(20:5)}]-5,0) {$Q\trans AQ$};
\end{tikzpicture}
\end{center}
\caption{Illustration of the two perturbations on $A$.}
\label{fig:twoperturb}
\end{figure}
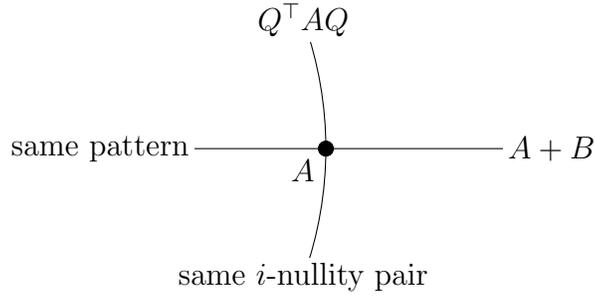

We consider the orbits of $A$ under the additive group action and the multiplicative group action in a small neighborhood of $A$. These can be viewed as geometric objects in the space $\msym$, as depicted in Figure~\ref{fig:twoperturb}.
Study the tangent spaces of these orbits naturally leads to the SNIP. More details about these techniques and related properties can be found in \cite{IEPGZF22}.

\begin{definition}
Let $U$ and $W$ be finite dimensional vector spaces over $\mathbb{R}$.  Let $F$ be a differentiable function from a domain of $U$ containing $\bu_0\in U$ to $W$.  Then the \emph{derivative} of $F$ at $\bu_0$ is a linear operator $\dot{F}$ defined by  
\[
    \dot{F}\cdot\bd = \lim_{t\rightarrow 0} \frac{F(\bu_0 + t\bd) - F(\bu_0)}{t}
\]
for each vector $\bd\in U$.
\end{definition}

This definition is equivalent to the classical {\it total derivative}, but gives a convenient notation for dealing with derivatives in multiple variables within matrix spaces.

\begin{example}
\label{ex:derivative}
Let $G$ be a graph on $n$ vertices,  $A\in\mptn(G)$, and $i\in V(G)$.  Let $F$ be the function  
\begin{equation}
\label{eq:supperturb}
    F(B,Q) = Q\trans AQ + B,
\end{equation}
defined for $B\in\mptncl(G)$ near $O$ and $Q\in\GL{i}$ near $I$,
and choose perturbation matrices $\Delta B \in \mptncl(G)$ and $\Delta Q \in \mat$ with $\Delta Q[i,i) = \bzero\trans$.  By direct computation, we have  
\[
    \begin{aligned}
    &\mathrel{\phantom{=}}\lim_{t\to 0}\frac{F(O + t\Delta B, I +t \Delta Q) - F(O,I)}{t} \\
    &= \lim_{t\to 0}\frac{(I + t\Delta Q)\trans A(I + t\Delta Q) + t\Delta B - A}{t} \\
    &= \lim_{t\to 0} \Delta Q\trans A + A\Delta Q + t\Delta Q\trans A\Delta Q + \Delta B \\
    & = \Delta Q\trans A + A\Delta Q + \Delta B.
    \end{aligned}
\]
The derivative of $F$ at $(O,I)$ is thus the linear operator defined by   \begin{equation}\label{eq:derivativeFatOI}
    \dot{F} \cdot (\Delta B, \Delta Q) = \Delta Q\trans A + A\Delta Q + \Delta B.
\end{equation}
This function can be specialized to a function of either argument by fixing the other:
Fixing $Q$ gives a map $F_B$ that takes $B$ to $Q\trans AQ +B$, and fixing $B$ gives a map $F_Q$ that takes $Q$ to $Q\trans AQ +B$.
Note that $\dot{F}_B$ and $\dot{F}_Q$ satisfy
\[
    \begin{aligned}
    \dot{F}_B \cdot \Delta B &= \Delta B, \\
    \dot{F}_Q \cdot \Delta Q &= \Delta Q\trans A + A\Delta Q.
    \end{aligned}
\]
This agrees with the chain rule  
\[
    \dot{F} \cdot (\Delta B, \Delta Q) = \dot{F}_B \cdot \Delta B + \dot{F}_Q \cdot \Delta Q.
\]
\end{example}

\begin{example}
Let $A$ be the all-ones matrix in $\mptn(K_2)$ and $i = 1$.  Consider the function $F$ from Equation~\eqref{eq:supperturb}.  Then we may write  
\[
    Q = \begin{bmatrix}
     x & 0 \\
     y & z
    \end{bmatrix}\text{ and }
    B = \begin{bmatrix}
     a & b \\
     b & c
    \end{bmatrix}.
\]
Thus, 
\[
    Q\trans AQ + B = \begin{bmatrix}
     x^2 + 2xy + y^2 + a & xz + yz + b \\
     xz + yz + b & z^2 + c
    \end{bmatrix}.
\]
Viewing $F$ as a function sending the six variables $x,y,z,a,b,c$ to the three independent entries in the symmetric matrix, we obtain that the total derivative of this map at $(x,y,z) = (1,0,1)$ and $(a,b,c) = (0,0,0)$ is  
\[
    \begin{bmatrix}
    2x + 2y & 2x + 2y & 0 & 1 & 0 & 0 \\
    z & z & x + y & 0 & 1 & 0 \\
    0 & 0 & 2z & 0 & 0 & 1 \\
    \end{bmatrix} = 
    \begin{bmatrix}
    2 & 2 & 0 & 1 & 0 & 0 \\
    1 & 1 & 1 & 0 & 1 & 0 \\
    0 & 0 & 2 & 0 & 0 & 1 \\
    \end{bmatrix}. 
\]
Note that this is the matrix representation of $\dot{F}$ by reading $\Delta Q$ with the basis $\{E_{1,1}, E_{2,1}, E_{2,2}\}$ and $\Delta B$ with the basis $\{E_{1,1}, E_{1,2} + E_{2,1}, E_{2,2}\}$. Therefore, considering the derivative $\dot{F}$ as a linear operator avoids the hassle of choosing the basis. 
\end{example}

The following proposition considers the tangent spaces and the normal spaces.  Here we use the inner product of matrices $\inp{A}{B} = \frac{1}{2}\tr(B\trans A)$ for symmetric matrices.  Note that the scalar $\frac{1}{2}$ does not change the orthogonality, but it has the benefit of giving $\|E_{i,j} + E_{j,i}\| = 1$ when $i \neq j$.  

\begin{proposition}
\label{prop:snipspaces}
Let $G$ be a graph on $n$ vertices, $A\in\mptn(G)$, and $i\in V(G)$. Let $F$ be defined as in Equation~\eqref{eq:supperturb},  and let $\dot{F}_B$, $\dot{F}_Q$ be the partial derivatives at $B = O$ and $Q = I$.  Then the following hold:
\begin{enumerate}[label={{\rm(\alph*)}}]
\item $\range(\dot{F}_B) = \mptncl(G)$.
\item $\range(\dot{F}_Q) = \{L\trans A + AL: L\in\mat,\ L[i,i) = \bzero\trans\}$.
\item $\range(\dot{F}_B)^\perp = \{X\in\msym: A\circ X = O,\ I\circ X = O\}$.
\item $\range(\dot{F}_Q)^\perp = \{X\in\msym: (AX)[i,i] = 0,\ (AX)(i,:] = O\}$.
\end{enumerate}
\end{proposition}
\begin{proof}
Note that every matrix in $\GL{i}$ near $I$ can be written as $I + L$ for some $L \in \mat$ with $L[i,i) = \bzero\trans$.  The ranges of $\dot{F}_B$ and $\dot{F}_Q$ are straightforward from Example~\ref{ex:derivative}.

For $\mptncl(G)$, since the diagonal entries and the entries corresponding to an edge are free, the orthogonal complement of $\mptncl(G)$ has free entries on non-edges.

Next we consider the $\range(\dot{F}_Q)$, which contains all symmetric matrices $X$ such that $\inp{X}{LA + AL} = 0$ for all $L\in\mat$ with $L[i,i) = \bzero\trans$.  By direct computation, we have  
\[
    \begin{aligned}
    2\inp{X}{L\trans A + AL} &= \tr(ALX) + \tr(L\trans AX) \\
    &= \tr(XL\trans A) + \tr(L\trans AX) \\
    &= \tr(L\trans AX) + \tr(L\trans AX) \\
    &= 4\inp{AX}{L},
    \end{aligned}
\]
which is zero only for all such $L$ when $(AX)[i,i] = 0$ and $(AX)(i,:] = O$. 
\end{proof}

With the above set up, now we are ready to justify the definition of the SNIP.  

\begin{proposition}
\label{prop:snipequiv}
Let $G$ be a graph on $n$ vertices, $A\in\mptn(G)$, and $i\in V(G)$.  Let $F$ be defined as in Equation~\eqref{eq:supperturb} and let $\dot{F}$, $\dot{F}_B$, $\dot{F}_Q$ be the derivative and the partial derivatives at $B = O$ and $Q = I$.  Then the following are equivalent:
\begin{enumerate}[label={{\rm(\alph*)}}]
\item $A$ has the $i$-SNIP.
\item $\dot{F}$ is surjective.
\item $\range(\dot{F}) = \range(\dot{F}_B) + \range(\dot{F}_Q) = \msym$.
\item $\range(\dot{F}_B)^\perp \cap \range(\dot{F}_Q)^\perp = \{O\}$.
\end{enumerate}
\end{proposition}
\begin{proof}
From Equation \eqref{eq:derivativeFatOI} we know  
\[
    \dot{F} \cdot (\Delta B, \Delta Q) = \dot{F}_B \cdot \Delta B + \dot{F}_Q \cdot \Delta Q.
\]
 Hence $\range(\dot{F}) = \range(\dot{F}_B) + \range(\dot{F}_Q)$, so statements (b) and (c) are equivalent by definition.  The equivalence between statements (c) and (d) follow from basic linear algebra facts about subspaces and their orthogonal complements.  Finally, we show that statements (d) and (a) are equivalent.  Observing $\range(\dot{F}_B)^\perp$ and $\range(\dot{F}_Q)^\perp$ in Proposition~\ref{prop:snipspaces}, we notice that $A\circ X = O$ and $I\circ X = O$ already implies $(AX)[i,i] = 0$, so the intersection of these two subspaces is  
\[
    \{X\in\msym: A\circ X = O,\ I\circ X = O, (AX)(i,:] = O\}.
\]
Therefore, the intersection is trivial if and only if $A$ has the $i$-SNIP.
\end{proof}

When the orbits of $A$ under $B$ and $Q$ (see Figure~\ref{fig:twoperturb}) are locally manifolds, 
we say that an intersection at $A$ is a \emph{transversal intersection} when it satisfies
the property that the span of the tangent spaces of those manifolds at $A$ is the entire ambient space.
However, we note that the proofs in this paper do not require the orbits to be manifolds.

Now that we have the basic background on the SNIP, we move our attention to the inverse function theorem, which will be the key ingredient in the proofs of the Minor Monotonicity (Theorem~\ref{thm:minmono}).  Note that the inverse function theorem is often stated for maps with bijective derivatives. Here we include a surjective version of the inverse function theorem; see, e.g., \cite{IEPGZF22}.

\begin{theorem}[Inverse Function Theorem]
\label{thm:invftsur}
Let $U$ and $W$ be finite-dimensional vector spaces over $\mathbb{R}$. Let $F$ be a smooth function from an open subset of $U$ to $W$ with $F(\bu_0) = \bw_0$. If the derivative $\dot{F}$ at $\bu_0$ is surjective, then there is an open subset $W'\subseteq W$ containing $\bw_0$ and a smooth function $T:W'\rightarrow U$ such that $T(\bw_0) = \bu_0$ and $F\circ T$ is the identity map on $W'$. 
\end{theorem}

\subsection{Supergraph lemma}
This section will be devoted to show the (extended) supergraph lemma. We start proving the following basic version of the supergraph lemma.

\begin{restatable}[Supergraph lemma]{lemma}{suplem}
\label{lem:supergraph}
Let $G$ be a spanning subgraph of $H$ and $i\in V(G)$.  Suppose $A\in\mptn(G)$ has the $i$-SNIP.  Then there is a matrix $A'\in\mptn(H)$ such that $A$ and $A'$ have the same $i$-nullity pair, $A'$ has the $i$-SNIP, and $A'$ can be chosen arbitrarily close to $A$.
\end{restatable}

We will use the function $F$ as in Equation~\eqref{eq:supperturb}.  As we will see in the proof, the SNIP and the inverse function theorem will guarantee the existence of $B'$ and $Q'$ such that $M = F(B',Q')$ for any $M\in\mptn(H)$ nearby $A$.  Thus, the matrix $A' = F(O,Q')$ is the desired matrix with the same $i$-nullity pair as $A$ and with the correct pattern.  The proof is illustrated in Figure~\ref{fig:supergraph}.

\begin{figure}[h]
\begin{center}
\begin{tikzpicture}
\clip (-4.5,-2) rectangle (4.5,2);
\draw (-3,0) -- (3,0);
\draw ([shift={(-40:5)}]-5,0) arc (-40:40:5);
\node[rectangle, fill=white, draw=none, text=black] at (-3,0) {same pattern}; 
\node[rectangle, fill=white, draw=none, text=black] at ([shift={(-20:5)}]-5,0) {same $i$-nullity pair};

\coordinate (a) at (0,0);
\node[label={225:$A$}, fill] (A) at (a) {};

\node[rectangle, fill=white, draw=none, text=black] at (3,0) {$A + B$}; 
\node[rectangle, fill=white, draw=none, text=black] at ([shift={(20:5)}]-5,0) {$Q\trans AQ$};

\coordinate (a1p) at ([shift={(10:5)}]-5,0);
\coordinate (m1) at ([xshift=-1cm]a1p); 

\draw[postaction={decorate}, decoration={markings, mark=at position 0.6 with {\arrow{>}}}, very thick, black] (a) arc (0:10:5);
\draw[postaction={decorate}, decoration={markings, mark=at position 0.6 with {\arrow{>}}}, very thick, black] (a1p) -- (m1);
\node[label={right:$A' = F(O,Q')$}, fill] (A1p) at (a1p) {};
\node[label={left:$M = F(B',Q')$}, fill] (M1) at (m1) {};
\end{tikzpicture}
\end{center}
\caption{Illustration of the proof of Lemma~\ref{lem:supergraph}.}
\label{fig:supergraph}
\end{figure}
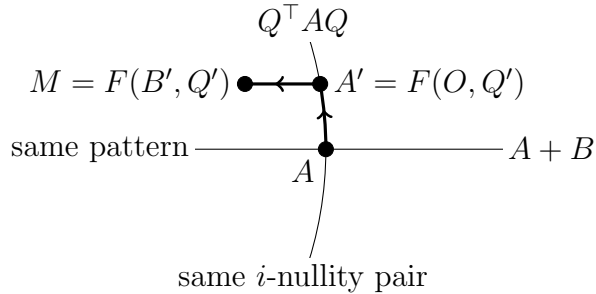

\begin{proof}[Proof of Lemma~\ref{lem:supergraph}]
Let $F$ be as defined in Equation~\eqref{eq:supperturb} and let $\dot{F}$ be its derivative at $B = O$ and $Q = I$.  Since $A$ has the $i$-SNIP, $\dot{F}$ is surjective by Proposition \ref{prop:snipequiv}.  By Theorem~\ref{thm:invftsur}, for each matrix $M\in\msym$ close enough to $A$, there are matrices $B'$ and $Q'$ such that $F(B',Q') = M$.  Since $G$ is a spanning subgraph of $H$, we may change the entries corresponding to $E(H)\setminus E(G)$ in $A$ into small nonzero values to obtain a matrix $M\in\mptn(H)$.  Although $M$ might not have the same $i$-nullity pair, the matrix $A' = F(O,Q')$ satisfies $A' = (Q')\trans AQ'$ and $A' = M - B'$.  When $B'$ and $Q'$ are small perturbations, $A'$ has the same pattern as $M$, i.e., $A'\in\mptn(H)$, and $A'$ has the same $i$-nullity pair as $A$.  

Since the perturbation is small, we still have   
\[
    \{L\trans A' + A'L: L\in\mat,\ L[i,i) = \bzero\trans\} + \mptncl(G) = \msym.
\]
Since $\mptncl(G)\subseteq\mptncl(H)$, the matrix $A'$ has the $i$-SNIP if we choose $M$ close enough to $A$.
\end{proof}

\begin{remark}
In the proof of Lemma~\ref{lem:supergraph}, since  $A'$ and $A$ are congruent matrices, they share the same inertia.
\end{remark}

\begin{restatable}[Extended supergraph lemma]{lemma}{exsuplem}
\label{cor:exsupergraph}
Let $G$ be a subgraph of $H$ and $i\in V(G)$.  Suppose $A\in\mptn(G)$ has the $i$-SNIP.  Then there is a matrix $A'\in\mptn(H)$ such that $A$ and $A'$ have the same $i$-nullity pair, and $A'$ has the $i$-SNIP.
\end{restatable}


\begin{proof}
Let $h = |V(H)| - |V(G)|$.  Let $A$ be a matrix in $\mptn(G)$ with the $i$-SNIP.  Then $A\oplus I_h$ is a matrix with the same $i$-nullity pair as $A$ with the $i$-SNIP.  By applying the supergraph lemma (Lemma~\ref{lem:supergraph}) to graphs $G\dunion\overline{K_h}$ and $H$ using the matrix $A\oplus I_h \in \mptn(G\dunion\overline{K_h})$, we obtain a matrix $A'$ with the $i$-SNIP and the same $i$-nullity pair as both $A\oplus I_h$ and $A$.
\end{proof}

\subsection{Decontraction lemma}
This section is devoted to proving the following decontraction lemma.

\begin{restatable}[Decontraction lemma]{lemma}{declem}
\label{lem:decontraction}
Let $G$ be a graph that is obtained from $H$ by contraction along the edge $\{u,v\}$ such that $u$ and $v$ have disjoint neighborhoods.  Suppose $A\in\mptn(G)$ has the $i$-SNIP.  Then there is a matrix $A'\in\mptn(H)$ such that $A$ and $A'$ have the same $i$-nullity pair, and $A'$ has the $i$-SNIP.  If $i\in V(G)$ is the new vertex through the contraction of $\{u,v\}$, then $i\in V(H)$ can be designated as either $u$ or $v$.
\end{restatable}

Throughout this subsection, we will assume that $G$ has $n$ vertices, $u = n$, $v = n+1$, and $n$ is the new vertex obtained by contracting $\{u,v\}$.  Note that $i$ can be any vertex in $[n]$. 
 Without loss of generality, we may assume $i \neq v$.  Let $\alpha = N_H(u)\setminus\{v\}$ and $\beta = N_H(v)\setminus\{u\}$.  By partitioning $[n]$ as $([n-1]\setminus (\alpha\cup\beta))$, $\alpha$, $\beta$, and $\{n\}$, we may write a matrix $A\in\mptn(G)$ as  
\begin{equation}
\label{eq:decA}
\renewcommand{\arraystretch}{1.2}
    A =
    \left[\begin{array}{ccc|c}
      &  &  & \bzero \\
      & A(n) &  & \ba \\
      &  &  & \bb \\
     \hline
     \bzero\trans & \ba\trans & \bb\trans & ?
    \end{array}\right],
\end{equation}
where we adopt the convention that $?$ indicates a number that may or may not be zero,
and $*$ indicates a nonzero number.
The proof involves two steps (first obtain $\tilde{A}$ and secondly obtain $A'$, see Figure \ref{fig:decontraction}). 

The first step is to perturb the intersection point $A_1 := A\oplus\begin{bmatrix} 1 \end{bmatrix}$ into a matrix  
\begin{equation}
\label{eq:decAtilda}
\renewcommand{\arraystretch}{1.2}
    \tilde{A} = \left[\begin{array}{ccc|cc}
      &  &  & \bzero & \bzero \\
      & \tilde{A}(\{n,n+1\}) &  & \tilde\ba & \bzero \\
      &  &  & \tilde\bb & \delta\tilde\bb \\
      \hline
     \bzero\trans & \tilde\ba\trans & \tilde\bb\trans & ? & 0 \\
     \bzero\trans & \bzero\trans & \delta\tilde\bb\trans & 0 & * \\
    \end{array}\right]
\end{equation}
with the same $i$-nullity pair.

The second step is to apply symmetric row and column operations to $\tilde{A}$ in order to obtain  
\begin{equation}
\label{eq:decAprime}
\renewcommand{\arraystretch}{1.2}
    A' = E\trans\tilde{A}E = \left[\begin{array}{ccc|cc}
      &  &  & \bzero & \bzero \\
      & \tilde{A}(\{n,n+1\}) &  & \tilde\ba & \bzero \\
      &  &  & \bzero & \delta\tilde\bb \\
      \hline
     \bzero\trans & \tilde\ba\trans & \bzero\trans & ? & * \\
     \bzero\trans & \bzero\trans & \delta\tilde\bb\trans & * & * \\
    \end{array}\right], 
\end{equation}
where  
\[
    E = I_{n-1} \oplus \left[\begin{array}{cc} 1 & 0\\ -\frac1\delta & 1 \end{array}\right].
\]
Note that $E\in\GL[n+1]{i}$ as $i\neq v$, and hence $A'$ and $\tilde{A}$ have the same $i$-nullity pair.  Figure~\ref{fig:decontraction} shows their relations as well as an overview of the proof.

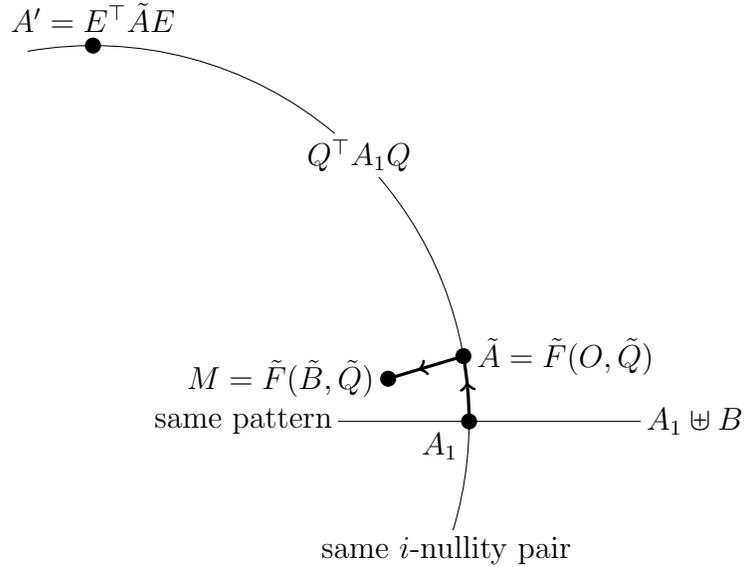
\begin{figure}[h]
\begin{center}
\begin{tikzpicture}
\clip (-6.5,-2) rectangle (4.5,6);
\draw (-3,0) -- (3,0);
\draw ([shift={(-40:5)}]-5,0) arc (-40:100:5);
\node[rectangle, fill=white, draw=none, text=black] at (-3,0) {same pattern}; 
\node[rectangle, fill=white, draw=none, text=black] at ([shift={(-20:5)}]-5,0) {same $i$-nullity pair};

\coordinate (a) at (0,0);
\node[label={225:$A_1$}, fill] (A) at (a) {};

\node[rectangle, fill=white, draw=none, text=black] at (3,0) {$A_1 \uplus B$}; 
\node[rectangle, fill=white, draw=none, text=black] at ([shift={(45:5)}]-5,0) {$Q\trans A_1Q$};

\coordinate (a1t) at ([shift={(10:5)}]-5,0);
\coordinate (a1p) at ([shift={(90:5)}]-5,0);
\coordinate (m1) at ([xshift=-1cm, yshift=-0.3cm]a1t); 

\draw[postaction={decorate}, decoration={markings, mark=at position 0.6 with {\arrow{>}}}, very thick, black] (a) arc (0:10:5);
\draw[postaction={decorate}, decoration={markings, mark=at position 0.6 with {\arrow{>}}}, very thick, black] (a1t) -- (m1);
\node[label={right:$\tilde{A} = \tilde{F}(O,\tilde{Q})$}, fill] (A1t) at (a1t) {};
\node[label={left:$M = \tilde{F}(\tilde{B},\tilde{Q})$}, fill] (M1) at (m1) {};
\node[label={above:$A' = E\trans \tilde{A}E$}, fill] (A1p) at (a1p) {};
\end{tikzpicture}
\end{center}
\caption{Illustration of the proof of Lemma~\ref{lem:decontraction}.}
\label{fig:decontraction}
\end{figure}

The first step to obtaining $\tilde{A}$ is very similar to the argument used for the supergraph lemma (Lemma~\ref{lem:supergraph}), except that here we will use a new perturbation function, which we denote by $\tilde{F}$.  Let $G_1 = G\dunion K_1$ be a spanning subgraph of $H$.  Then $A_1\in\mptn(G_1)$.  For each $B\in\mptn(G_1)$ and each real number $\delta$, define $B_\delta$ as the matrix obtained from $B$ by adding $\delta B[n,\beta]$ to $B[n+1,\beta]$ and $\delta B[\beta, n]$ to $B[\beta, n+1]$.  Thus, the set  
\[
    \mptn(G_1,\delta) := \{B_\delta: B\in\mptn(G_1)\}
\]
is a subspace with the same dimension as $\mptn(G_1)$.  Now we fix a vertex $w\in\beta$, and for each $C\in\msym[n+1]$ we define  
\[
    C\uplus B := C + B_\delta,
\]
where $\delta = \frac{C[n+1,w]}{C[n,w]}$ is a ratio defined by $C$.  Note that $C\uplus B$ is well-defined for each $C$ nearby $A_1$, and for $C\in\mptn(G_1)$, $C\uplus B$ is simply $C + B$.  Now let $\tilde{F}$ be the perturbation function   
\begin{equation}
\label{eq:decperturb}
    \tilde{F}(B,Q) = (Q\trans A_1Q) \uplus B
\end{equation}
defined on $B\in\mptncl(G_1)$ nearby $O$ and $Q\in\GL[n+1]{i}$ nearby $I$.  Since $\tilde{F}(B, I) = A_1 + B$, the calculation of the derivative of $\tilde{F}$ is almost the same as that of $F$ in Equation~\eqref{eq:supperturb}, and its connection to the SNIP is the same. Thus we have the following remark.

\begin{remark}
\label{rem:A1snipequiv}
Let $\tilde{F}$ be the perturbation function defined in Equation~\eqref{eq:decperturb}, and let $\dot{\tilde{F}}$ be its derivative.  Then $A_1$ has the $i$-SNIP if and only if $\dot{\tilde{F}}$ is surjective.
\end{remark}

Now we are ready to formalize the aforementioned arguments.

\begin{lemma}
\label{lem:parallel}
Let $G$ be a graph on $n$ vertices, $i\in V(G)$, $\beta\subseteq N_G(n)$, and $G_1 = G\dunion K_1$.  Suppose $A\in\mptn(G)$ has the $i$-SNIP.  Then there is a matrix $\tilde{A}$ of the form \eqref{eq:decAtilda} such that $A_1 = A\oplus\begin{bmatrix} 1 \end{bmatrix}$ and $\tilde{A}$ have the same $i$-nullity pair, $\tilde{A}$ has the property  
\[
    \{L\trans \tilde{A} + \tilde{A}L: L\in\mat[n+1],\ L[i,i) = \bzero\trans\} + \mptncl(G_1,\delta) = \msym[n+1],
\]
and $A'$ can be chosen arbitrarily close to $A_1$ with $\delta > 0$.  
\end{lemma}
\begin{proof}
The proof argument is essentially the same as the proof of Lemma~\ref{lem:supergraph}.  Since $A$ has the $i$-SNIP, from Proposition~\ref{prop:directsum} it follows that $A_1 = A\oplus\begin{bmatrix} 1 \end{bmatrix}$ also has the $i$-SNIP.  Let $\tilde{F}$ be the perturbation function defined in Equation~\eqref{eq:decperturb}, and let $\dot{\tilde{F}}$ be its derivative at $B = O$ and $Q = I$.  Thus, $\dot{\tilde{F}}$ is surjective by Remark~\ref{rem:A1snipequiv}.  By Theorem~\ref{thm:invftsur}, for each matrix $M\in\msym[n+1]$ close enough to $A_1$, there are matrices $\tilde{B}$ and $\tilde{Q}$ such that $\tilde{F}(\tilde{B}, \tilde{Q}) = M$.  

Now, for a small $\delta > 0$, choose the matrix $M$ as the matrix obtained from $A_1$ by adding $\delta A_1[n, \beta]$ to $A_1[n+1, \beta]$ and $\delta A_1[\beta, n]$ to $A_1[\beta, n+1]$.  That is, $M = (A_1)_\delta$.  Although $M$ need not have the same $i$-nullity pair, the matrix $\tilde{A} = \tilde{F}(O,\tilde{Q})$ satisfies $A' = \tilde{Q}\trans A_1\tilde{Q}$ and $\tilde{A} = M - \tilde{B}_\delta$.  When $\tilde{B}$ and $\tilde{Q}$ are small perturbations, $\tilde{A}$ has the form \eqref{eq:decAtilda}, and $\tilde{A}$ has the same $i$-nullity pair as $A_1$.  

Since the perturbation $\delta$ is small, we still have   
\[
    \{L\trans \tilde{A} + \tilde{A}L: L\in\mat,\ L[i,i) = \bzero\trans\} + \mptncl(G_1,\delta) = \msym[n+1].
\]
This completes the proof.
\end{proof}

The choice of the perturbation function $\tilde{F}$ successfully makes the two rows of $\tilde{A}[\{n,n_1\},\beta]$ parallel and makes $\tilde{A}$ possess a condition similar to the $i$-SNIP, which will be transformed into $i$-SNIP later.  Now the proof of Lemma~\ref{lem:decontraction} is just one step away.  

\begin{proof}[Proof of Lemma~\ref{lem:decontraction}]
Let $\tilde{A}$ be the matrix guaranteed by Lemma~\ref{lem:parallel} and 
\[
    E = I_{n-1} \oplus \left[\begin{array}{cc} 1 & 0\\ -\frac1\delta & 1 \end{array}\right].
\]
Then $A' = E\trans \tilde{A}E$ has the form \eqref{eq:decAprime}.  Since $\tilde{A}$ can be chosen arbitrarily close to $A_1$, we have $A'\in\mptn(H)$.  Since $E\in\GL[n+1]{i}$, the matrices $A'$, $\tilde{A}$, and $A_1$ have the same $i$-nullity pair.

Finally, $C\mapsto E\trans CE$ is a linear bijection from $\msym[n+1]$ to $\msym[n+1]$ as well as a bijection on
the space
\[
\mathcal{L} :=
\{L\in\mat[n+1],\ L[i,i) = \bzero\trans\}.
\]
By applying this bijection to each term in  
\[
    \{L\trans \tilde{A} + \tilde{A}L: L\in\mathcal{L} \}+ \mptncl(G_1,\delta) = \msym[n+1],
\]
we know that the span of   
\[
    \{E\trans L\trans (E\trans)^{-1}A' + A'E^{-1}LE: L\in\mathcal{L}\} \text{ and } \mptncl(H)
\]
is $\msym[n+1]$.  Additionally, since $L\mapsto E^{-1}LE$ is a linear bijection from 
$\mathcal{L}$
to itself, we get  
\[
    \{L\trans A' + A'L: L\in\mathcal{L}\} + \mptncl(H) = \msym[n+1],
\]
implying $A'$ has the $i$-SNIP.
\end{proof}

Combining the extended supergraph lemma (Lemma~\ref{cor:exsupergraph}) and the decontraction lemma (Lemma~\ref{lem:decontraction}), we obtain the desired minor monotonicity result.  

\minmono*


\subsection*{Acknowledgements}
This project started and was made possible by the 2021 IEPG-ZF Virtual Research Community at the American Institute of Mathematics (AIM), with support from the US National Science Foundation. 
Aida Abiad is partially supported by the Dutch Research Council via the grant VI.Vidi.213.085. The research of B. Curtis was partially supported by NSF grant 1839918. Mary Flagg and the AIM SQuaRE where this work was completed were partially supported by grant DMS 2331634 from the National Science Foundation. Jephian C.-H. Lin was supported by the Young Scholar Fellowship Program (grant no.\ NSTC-112-2628-M-110-003) from the National Science and Technology Council of Taiwan.



\end{document}